\documentclass[11pt]{article}
 
\usepackage[margin=0.5in]{geometry} 
\usepackage{amsmath,amsthm,amssymb}
\usepackage{mathrsfs}
\usepackage[T1]{fontenc}
\usepackage[english]{babel}
\usepackage{graphicx}
\usepackage{color} 

\usepackage{hyperref}

\usepackage{color}

\usepackage{mathtools}
\usepackage{atbegshi}
\usepackage{pdfpages} 
\usepackage{longtable} 
\usepackage{amsmath,amsthm,amsfonts,amssymb,relsize,bm} 
\usepackage{xparse,xcolor} 
\usepackage{graphicx,float} 
\usepackage{etoolbox}
\usepackage[shortlabels]{enumitem} 
\usepackage{tikz-cd} 
\usepackage{cite} 
\usepackage[normalem]{ulem}
\usepackage{comment}
\usepackage{tocloft}

\numberwithin{equation}{section}

\theoremstyle{plain}
\newtheorem{theorem}{Theorem}[section]

\newtheorem*{theorem*}{Theorem}
\newtheorem{lemma}{Lemma}[section]
\newtheorem{corollary}{Corollary}[section]

\theoremstyle{definition}
\newtheorem{definition}{Definition}[section]
\newtheorem*{definition*}{Definition}

\newtheorem{remark}{Remark}[section]

\begin{document}
 
 \title{An analogue of a formula of Popov II}

\author{Pedro Ribeiro}

\date{}
 
\maketitle

\begin{abstract} Let $r_{k}(n)$ denote the number of representations of the positive
integer $n$ as the sum of $k$ squares. We prove a generalization of a summation formula already proved by us [Advances in Applied Mathematics, 175 (2026) 103201], which involves the arithmetical function $r_{k}(n)$ and the Bessel functions of the first kind. We extend the Bessel functions in the aforementioned formula to Whittaker functions, and our proof of this generalization is drastically different from the proof of the particular case presented in [Advances in Applied Mathematics, 175 (2026) 103201]. 
\end{abstract}

\section{Introduction and Main results}
In a fairly unknown paper \cite{popov_russian}, Popov states the following beautiful result. If $\text{Re}(x)>0$ and $z\in\mathbb{C}$, then
\begin{align}
\frac{z^{\frac{k}{2}-1}\pi^{\frac{k}{4}-\frac{1}{2}}x^{\frac{k}{4}}}{2^{\frac{k}{2}-1}\Gamma\left(\frac{k}{2}\right)}\,e^{z^{2}/8}+\sqrt{x}\,e^{z^{2}/8}\,\sum_{n=1}^{\infty}\frac{r_{k}(n)}{n^{\frac{k}{4}-\frac{1}{2}}}\,e^{-\pi nx}\,J_{\frac{k}{2}-1}(\sqrt{\pi nx}z)\nonumber \\
=\frac{z^{\frac{k}{2}-1}\pi^{\frac{k}{4}-\frac{1}{2}}x^{-\frac{k}{4}}}{2^{\frac{k}{2}-1}\Gamma\left(\frac{k}{2}\right)}\,e^{-z^{2}/8}+\frac{e^{-z^{2}/8}}{\sqrt{x}}\,\sum_{n=1}^{\infty}\frac{r_{k}(n)}{n^{\frac{k}{4}-\frac{1}{2}}}\,e^{-\frac{\pi n}{x}}\,I_{\frac{k}{2}-1}\left(\sqrt{\frac{\pi n}{x}}z\right),\label{Popov intro}
\end{align}
where $r_{k}(n)$ denotes the number of representations of the positive
integer $n$ as a sum of $k$ squares and, as usual, $J_{\nu}(z)$
and $I_{\nu}(z)$ respectively denote the Bessel and modified Bessel
functions of the first kind. 
A couple of reasons why this identity
is fascinating are already provided by Berndt, Dixit, Kim and Zaharescu {[}\cite{berndt_popov}, pp. 3795-3796{]}. For the purposes of our discussion, we repeat verbatim the reasons already explained in \cite{RPOPOV}. 
\begin{enumerate}
\item If we construct the Dirichlet series attached to $r_{k}(n)$, 
\begin{equation}
\zeta_{k}(s)=\sum_{n=1}^{\infty}\frac{r_{k}(n)}{n^{s}},\,\,\,\,\text{Re}(s)>\frac{k}{2},\label{definition zeta k}
\end{equation}
then $\zeta_{k}(s)$ can be continued to the complex plane as a meromorphic
function possessing only a simple pole at $s=\frac{k}{2}$ with residue
$\pi^{k/2}/\Gamma(k/2)$. Moreover, it satisfies the functional equation
\begin{equation}
\pi^{-s}\Gamma\left(s\right)\zeta_{k}(s)=\pi^{s-\frac{k}{2}}\Gamma\left(\frac{k}{2}-s\right)\zeta_{k}\left(\frac{k}{2}-s\right).\label{functional equation zetak}
\end{equation}
Note that, when $k=1$, $r_{1}(n)=2$ if and only if $n$ is a perfect
square and zero otherwise. Therefore, (\ref{definition zeta k}) reduces
to 
\begin{equation}
\zeta_{1}(s):=\sum_{n=1}^{\infty}\frac{r_{1}(n)}{n^{s}}=2\sum_{n=1}^{\infty}\frac{1}{n^{2s}}=2\zeta(2s),\,\,\,\,\text{Re}(s)>\frac{1}{2}.\label{zeta 1(s) definition}
\end{equation}
Furthermore, (\ref{functional equation zetak}) with $k=1$ gives the functional equation
for Riemann's $\zeta-$function
\begin{equation}
\pi^{-s}\Gamma\left(s\right)\zeta(2s)=\pi^{s-\frac{1}{2}}\Gamma\left(\frac{1}{2}-s\right)\zeta\left(1-2s\right).\label{Functional equation Riemann Popov}
\end{equation}
The first point highlighted in \cite{berndt_popov} is that the powers of $n$
in the denominators of both sides of (\ref{Popov intro}) are remindful
of the functional equation (\ref{functional equation zetak}).
\item Riemann's second proof of the functional equation for $\zeta(s)$,
(\ref{Functional equation Riemann Popov}), employs the transformation
formula for Jacobi's $\theta-$function,
\begin{equation}
\theta(x):=\sum_{n\in\mathbb{Z}}e^{-\pi n^{2}x}=\frac{1}{\sqrt{x}}\sum_{n\in\mathbb{Z}}e^{-\frac{\pi n^{2}}{x}}:=\frac{1}{\sqrt{x}}\theta\left(\frac{1}{x}\right),\,\,\,\,\text{Re}(x)>0.\label{jacobi theta intro popov}
\end{equation}
The theta transformation formula associated to the Dirichlet series
$\zeta_{k}(s)$ can be obtained by taking the $k^{\text{th}}$ power
on both sides of (\ref{jacobi theta intro popov}). This results in
the transformation
\begin{equation}
\sum_{n=0}^{\infty}r_{k}(n)\,e^{-\pi nx}=x^{-\frac{k}{2}}\sum_{n=0}^{\infty}r_{k}(n)\,e^{-\frac{\pi n}{x}},\,\,\,\,\text{Re}(x)>0.\label{theta k squares formula}
\end{equation}
Of course, the exponential factors on both sides of (\ref{Popov intro})
remind us the theta transformation formula (\ref{theta k squares formula}).
In fact, (\ref{theta k squares formula}) is a particular case of
(\ref{Popov intro}) when we let $z\rightarrow0$, due to the limiting relations for the Bessel functions {[}\cite{NIST}, p. 223,
eq. (10.7.3){]}
\begin{equation}
\lim_{y\rightarrow0}y^{-\nu}J_{\nu}(y)=\lim_{y\rightarrow0}y^{-\nu}I_{\nu}(y)=\frac{2^{-\nu}}{\Gamma(\nu+1)}. \label{limiting Bessel}
\end{equation}
\item Chandrasekharan and Narasimhan {[}\cite{arithmetical identities}, p. 19, eq. (65){]} proved yet
another equivalent identity to (\ref{functional equation zetak})
and (\ref{theta k squares formula}). If $x>0$ and $q>\frac{k-1}{2}$,
then
\begin{equation}
\frac{1}{\Gamma(q+1)}\,\sum_{0\leq n\leq x}{}^{^{\prime}}r_{k}(n)\,(x-n)^{q}=\frac{\pi^{\frac{k}{2}}x^{\frac{k}{2}+q}}{\Gamma\left(q+1+\frac{k}{2}\right)}+\pi^{-q}\sum_{n=1}^{\infty}r_{k}(n)\left(\frac{x}{n}\right)^{\frac{k}{4}+\frac{q}{2}}J_{\frac{k}{2}+q}\left(2\pi\sqrt{nx}\right),\label{Bessel expansion riesz sums}
\end{equation}
where the Bessel series on the right-hand side converges absolutely.
The prime on the summation sign indicates that, if $q=0$ and $x$
is an integer, then the last contribution in this Riesz sum is just
$\frac{1}{2}r_{k}(x)$. The appearance of the Bessel functions in
(\ref{Popov intro}) reminds us of (\ref{Bessel expansion riesz sums}).
\end{enumerate}

In \cite{RPOPOV} we have proved two interesting summation formulas that extend the transformation formula for Jacobi's theta function (\ref{jacobi theta intro popov}). These formulas are, respectively, 
\begin{align}
\frac{(\pi y)^{\frac{k}{4}-\frac{1}{2}}}{2^{\frac{k}{4}-\frac{1}{2}}\Gamma\left(\frac{k}{4}+\frac{1}{2}\right)}&+\sum_{n=1}^{\infty}\frac{r_{k}(n)}{n^{\frac{k}{4}-\frac{1}{2}}}\,e^{-\pi nx}J_{\frac{k}{4}-\frac{1}{2}}\left(\pi ny\right)\nonumber \\
=\frac{(\pi y)^{\frac{k}{4}-\frac{1}{2}}}{2^{\frac{k}{4}-\frac{1}{2}}\Gamma\left(\frac{k}{4}+\frac{1}{2}\right)\left(x^{2}+y^{2}\right)^{\frac{k}{4}}}&+\frac{1}{\sqrt{x^{2}+y^{2}}}\,\sum_{n=1}^{\infty}\frac{r_{k}(n)}{n^{\frac{k}{4}-\frac{1}{2}}}\,e^{-\frac{\pi nx}{x^{2}+y^{2}}} J_{\frac{k}{4}-\frac{1}{2}}\left(\frac{\pi ny}{x^{2}+y^{2}}\right)\label{Formula to prove analogue Popov in chapter 6}
\end{align}
and 
\begin{align}
\frac{(\pi y)^{\frac{k}{4}-\frac{1}{2}}}{2^{\frac{k}{4}-\frac{1}{2}}\Gamma\left(\frac{k}{4}+\frac{1}{2}\right)}&+\ensuremath{\sum_{n=1}^{\infty}\frac{r_{k}(n)}{n^{\frac{k}{4}-\frac{1}{2}}}\,e^{-\pi nx}I_{\frac{k}{4}-\frac{1}{2}}\left(\pi ny\right)}\nonumber \\
=\frac{(\pi y)^{\frac{k}{4}-\frac{1}{2}}}{2^{\frac{k}{4}-\frac{1}{2}}\Gamma\left(\frac{k}{4}+\frac{1}{2}\right)\left(x^{2}-y^{2}\right)^{\frac{k}{4}}}&+\frac{1}{\sqrt{x^{2}-y^{2}}}\sum_{n=1}^{\infty}\frac{r_{k}(n)}{n^{\frac{k}{4}-\frac{1}{2}}}\,e^{-\frac{\pi nx}{x^{2}-y^{2}}} I_{\frac{k}{4}-\frac{1}{2}}\left(\frac{\pi ny}{x^{2}-y^{2}}\right).\label{formula for modified bessel in chapter 6}
\end{align}

Before proceeding to the main result of this paper, let us remark that (\ref{Formula to prove analogue Popov in chapter 6}) and (\ref{formula for modified bessel in chapter 6}) contain some interesting analogues and particular cases. First, if we take $k=4$ in (\ref{formula for modified bessel in chapter 6}), we obtain the curious identity
\begin{equation*}
2\pi y+\sum_{n=1}^{\infty}\frac{r_{4}(n)}{n}\,\left\{ e^{-\pi n(x-y)}-e^{-\pi n(x+y)}\right\}=\frac{2\pi y}{x^{2}-y^{2}}+\sum_{n=1}^{\infty}\frac{r_{4}(n)}{n}\,\left\{ e^{-\frac{\pi n}{x+y}}-e^{-\frac{\pi n}{x-y}}\right\},
\end{equation*}
valid for $x>y>0$. 
Also, when $k=1$, it is simple to see that (\ref{Formula to prove analogue Popov in chapter 6}) and (\ref{formula for modified bessel in chapter 6}) imply the beautiful formulas
\begin{equation}
\sum_{n\in\mathbb{Z}}\sqrt{|n|}\,e^{-\pi n^{2}x}J_{-\frac{1}{4}}\left(\pi n^{2}y\right)=\frac{1}{\sqrt{x^{2}+y^{2}}}\sum_{n\in\mathbb{Z}}\sqrt{|n|}\,e^{-\frac{\pi n^{2}x}{x^{2}+y^{2}}}\,J_{-\frac{1}{4}}\left(\frac{\pi n^{2}y}{x^{2}+y^{2}}\right),\label{formula k=00003D1 poppoooov case}
\end{equation}
\begin{equation}
\sum_{n\in\mathbb{Z}}\sqrt{|n|}\,e^{-\pi n^{2}x}I_{-\frac{1}{4}}\left(\pi n^{2}y\right)=\frac{1}{\sqrt{x^{2}-y^{2}}}\sum_{n\in\mathbb{Z}}\sqrt{|n|}\,e^{-\frac{\pi n^{2}x}{x^{2}-y^{2}}}\,I_{-\frac{1}{4}}\left(\frac{\pi n^{2}y}{x^{2}-y^{2}}\right),\label{formuuula k=00003D1 I modified Popov}
\end{equation}
where $x>y>0$. Of course, when $y\rightarrow0^{+}$, (\ref{formula k=00003D1 poppoooov case})
and (\ref{formuuula k=00003D1 I modified Popov}) both reduce to the
theta transformation formula (\ref{jacobi theta intro popov}), by virtue of the limiting relation (\ref{limiting Bessel}).  

Although the proofs of (\ref{Formula to prove analogue Popov in chapter 6}) and (\ref{formula for modified bessel in chapter 6}) employed the interesting symmetries of the Gauss hypergeometric function, $_{2}F_{1}\left(a,b;c;z\right)$, the reader may find their scope a bit limited. First of all, the condition $x>y>0$ seems to be unnecessary for the first formula, (\ref{Formula to prove analogue Popov in chapter 6}), to be valid. However, the proof developed in \cite{RPOPOV} only seems to work under this assumption. 
This particular limitation begs the question whether it is possible to obtain a generalization of (\ref{Formula to prove analogue Popov in chapter 6}) and (\ref{formula for modified bessel in chapter 6}) that lifts the restrictive conditions on $x$ and $y$ imposed by (\ref{Formula to prove analogue Popov in chapter 6}) and (\ref{formula for modified bessel in chapter 6}). 

Since we have already studied summation formulas for indices of Whittaker functions, we thought it would be suitable to complement the results from both \cite{RYPIT} and \cite{RPOPOV} by presenting a generalization of (\ref{Formula to prove analogue Popov in chapter 6}) and (\ref{formula for modified bessel in chapter 6}) involving the Whittaker function of the first kind, $M_{\mu, \nu} (z)$. As it is well known, this function arises in the study of the Kummer
confluent hypergeometric function, $_{1}F_{1}(a;c;z)$, which is usually
defined by the power series [\cite{NIST}, p. 322, eq. (13.2.2)],
\begin{equation}
_{1}F_{1}\left(a;c;z\right):=\sum_{k=0}^{\infty}\frac{(a)_{k}}{(c)_{k}}\,\frac{z^{k}}{k!}, \label{kummer def RYPIT chapter}
\end{equation}
where $(a)_{k}:=\Gamma(a+k)/\Gamma(a)$ denotes the Pochhammer symbol. From the power series (\ref{kummer def RYPIT chapter}), one can define the Whittaker function of the first kind as [\cite{NIST}, p. 334, eq. (13.14.2)]
\begin{equation}
M_{\mu,\nu}(z)=z^{\nu+\frac{1}{2}}e^{-z/2}\,_{1}F_{1}\left(\nu-\mu+\frac{1}{2};\,2\nu+1;\,z\right).\label{NIST def Whittaker first kind}
\end{equation}
This function has several interesting integral representations. One representation that could prove itself useful to be employed in a potential generalization of (\ref{Formula to prove analogue Popov in chapter 6}) and (\ref{formula for modified bessel in chapter 6}) is [\cite{handbook_marichev}, p. 458, eq. (3.30.2.1)] 
\begin{equation}
e^{-ax}M_{\rho,\nu}(2bx)=\left(\frac{2b}{a+b}\right)^{\nu}\,\frac{\sqrt{2bx}}{2\pi i}\,\intop_{\sigma-i\infty}^{\sigma+i\infty}\,\Gamma\left(s+\nu\right)\,_{2}F_{1}\left(\nu+\frac{1}{2}-\rho,s+\nu;2\nu+1;\frac{2b}{a+b}\right)\,\left(x\left(a+b\right)\right)^{-s}\,ds,\label{Marichev at end to apply in the proof}
\end{equation}
as its integrand possesses just the right kind of symmetries, compatible not only with the functional equation of $\zeta_{k}(s)$, but also with Euler's formula for Gauss' hypergeometric function. Applying the same ideas as the ones presented in \cite{RPOPOV}, one can rightfully come to the following result. 
\begin{theorem*} \label{weaker theorem}
Assume that $x,y$ are two real numbers such that $x>y>0$ and that
$-\frac{k}{4}<\rho<\frac{k}{4}$. Then the following summation formula
holds
\begin{align}
(2\pi y)^{\frac{k}{4}}+\sum_{n=1}^{\infty}\frac{r_{k}(n)}{n^{\frac{k}{4}}} & e^{-\pi nx}\,M_{\rho,\frac{k}{4}-\frac{1}{2}}\left(2\pi ny\right)\nonumber \\
=\left(\frac{2\pi y}{x^{2}-y^{2}}\right)^{\frac{k}{4}}\,\left(\frac{x-y}{x+y}\right)^{\rho}+ & \left(\frac{x-y}{x+y}\right)^{\rho}\,\sum_{n=1}^{\infty}\frac{r_{k}(n)}{n^{\frac{k}{4}}}e^{-\frac{\pi nx}{x^{2}-y^{2}}}\,M_{-\rho,\frac{k}{4}-\frac{1}{2}}\left(\frac{2\pi ny}{x^{2}-y^{2}}\right).\label{First summation formula final proof}
\end{align}
\end{theorem*}

Just like the conditions giving the formulas (\ref{Formula to prove analogue Popov in chapter 6}) and (\ref{formula for modified bessel in chapter 6}), the previous result offers significant obstacles to a wide extension of (\ref{Formula to prove analogue Popov in chapter 6}) and (\ref{formula for modified bessel in chapter 6}). Note that the condition $x>y>0$ is indeed crucial to apply Watson-type estimates for the Gauss hypergeometric function (cf. [\cite{RPOPOV}, Lemma 2.1.]), which are vital to prove (\ref{First summation formula final proof}) using the Mellin-Barnes integral (\ref{Marichev at end to apply in the proof}). However, the condition $x>y>0$ is way too restrictive because, when $\rho =0$, (\ref{First summation formula final proof}) can only be reduced to (\ref{formula for modified bessel in chapter 6}) and not to (\ref{Formula to prove analogue Popov in chapter 6})!
\bigskip{}

Therefore, we need an argument that provides a slightly more general version of (\ref{First summation formula final proof}) and yields, at the same time, both (\ref{Formula to prove analogue Popov in chapter 6}) and (\ref{formula for modified bessel in chapter 6}). 
Finding such an identity requires that we present a drastically different argument from the one used in \cite{RPOPOV}, not needing at all the implementation of the Mellin-Barnes integral (\ref{Marichev at end to apply in the proof}).
What is truly curious about this different proof is that it uses nothing more than Popov's classical formula (\ref{Popov intro}), which served as the main motivation for the investigation of (\ref{Formula to prove analogue Popov in chapter 6}) and (\ref{formula for modified bessel in chapter 6}) in the first place! Without any further delay, we present the main result of this paper. 

\begin{theorem}\label{generalization whittaker first}
Assume that $x,y$ are two complex numbers such that $\text{Re}(x)>|\text{Re}(y)|$ and that $\rho$ is a
complex number satisfying $-\frac{k}{4}<\text{Re}(\rho)<\frac{k}{4}$. Then the following summation formula holds
\begin{align}
(2\pi y)^{\frac{k}{4}}+\sum_{n=1}^{\infty}\frac{r_{k}(n)}{n^{\frac{k}{4}}} & e^{-\pi nx}\,M_{\rho,\frac{k}{4}-\frac{1}{2}}\left(2\pi ny\right)\nonumber \\
=\left(\frac{2\pi y}{x^{2}-y^{2}}\right)^{\frac{k}{4}}\,\left(\frac{x-y}{x+y}\right)^{\rho}+ & \left(\frac{x-y}{x+y}\right)^{\rho}\,\sum_{n=1}^{\infty}\frac{r_{k}(n)}{n^{\frac{k}{4}}}e^{-\frac{\pi nx}{x^{2}-y^{2}}}\,M_{-\rho,\frac{k}{4}-\frac{1}{2}}\left(\frac{2\pi ny}{x^{2}-y^{2}}\right),\label{Formula Formula Formula Whittaker GENERAL}
\end{align}
where $M_{\mu,\nu}(z)$ denotes the Whittaker function of the first
kind (\ref{NIST def Whittaker first kind}).
\end{theorem}

To end our introduction, let us remark that it comes as no surprise that we can extend Theorem \ref{generalization whittaker first} to a more general class
of Dirichlet series satisfying Hecke's functional equation. For example,
in \cite{ryce_published}, the author of this paper and Yakubovich proved a generalization of Popov's summation formula, (\ref{Popov intro}), with the role of $r_{k}(n)$ being replaced by a generic arithmetical function, $a(n)$.\footnote{To be completely fair, Berndt [\cite{dirichletserisV}, p. 154, Example 7] already writes this generalized version of Popov's formula as a particular case of an analogue of Vorono\"i's summation formula. Our proof in \cite{ryce_published}, on the other hand, had as its main motivation the study of the critical zeros of some combinations of Dirichlet series. It is, therefore, totally different from Berndt's, as it uses Kummer's formula for the confluent hypergeometric function, $_{1}F_{1}(a;\,c;\,z)$.} To work in this setting, we introduced the following general class of Dirichlet series (cf. [\cite{ryce_published}, Def. 1.1.]).

\begin{definition} \label{class A definition}
Let $\left(\lambda_{n}\right)_{n\in\mathbb{N}}$ and $\left(\mu_{n}\right)_{n\in\mathbb{N}}$
be two sequences of positive numbers strictly increasing to $\infty$
and $\left(a(n)\right)_{n\in\mathbb{N}}$ and $\left(b(n)\right)_{n\in\mathbb{N}}$
two sequences of complex numbers not identically zero. Consider the
functions $\phi(s)$ and $\psi(s)$ representable as Dirichlet series 
\begin{equation}
\phi(s)=\sum_{n=1}^{\infty}\frac{a(n)}{\lambda_{n}^{s}}\,\,\,\,\,\,\text{and }\,\,\,\,\,\psi(s)=\sum_{n=1}^{\infty}\frac{b(n)}{\mu_{n}^{s}}\label{representable as Dirichlet series in first definition ever}
\end{equation}
with finite abscissas of absolute convergence $\sigma_{a}$
and $\sigma_{b}$, respectively. We say that
$\phi(s)$ and $\psi(s)$ satisfy the functional equation \begin{equation}
\Gamma(s)\,\phi(s)=\Gamma(r-s)\,\psi(r-s),\,\,\,\,\,\,r>0,\label{Hecke Dirichlet series Functional}
\end{equation}
if there exists a meromorphic function $\chi(s)$ with the following
properties:
\begin{enumerate}
\item $\chi(s)=\Gamma(s)\,\phi(s)$ for $\text{Re}(s)>\sigma_{a}$
and $\chi(s)=\Gamma(r-s)\,\psi(r-s)$ for $\text{Re}(s)<r-\sigma_{b}$;
\item $\lim_{|\text{Im}(s)|\rightarrow\infty}\chi(s)=0$ uniformly in every
interval $-\infty<\sigma_{1}\leq\text{Re}(s)\leq\sigma_{2}<\infty$. 
\item $\phi(s)$ and $\psi(s)$ have analytic continuations to the entire complex plane and are analytic on $\mathbb{C}$ except
for possible simple poles located at $s=r$ with residues $\rho$
and $\rho^{\star}$, respectively.
\end{enumerate}
\end{definition}

\bigskip{}

\bigskip{}

Under the scope of the previous definition, we have been able to prove the following generalization of Popov's formula (\ref{Popov intro}) [\cite{ryce_published}, p. 38, eq. (2.64)], 
\begin{align}
-\phi(0)e^{z^{2}/8}+2^{r-1}\Gamma(r)\,x^{\frac{1-r}{2}}z^{1-r}\,e^{z^{2}/8}\,\sum_{n=1}^{\infty}a(n)\,\lambda_{n}^{\frac{1-r}{2}}e^{-\lambda_{n}x}\,J_{r-1}\left(\sqrt{\lambda_{n}x}\,z\right)\nonumber\\
=\frac{\rho\Gamma(r)}{x^{r}}e^{-z^{2}/8}+2^{r-1}\Gamma(r)\,z^{1-r}x^{-\frac{r+1}{2}}e^{-z^{2}/8}\,\sum_{n=1}^{\infty}b(n)\,\mu_{n}^{\frac{1-r}{2}}e^{-\frac{\mu_{n}}{x}}\,I_{r-1}\left(\sqrt{\frac{\mu_{n}}{x}}\,z\right),\label{popov ryce paper}
\end{align}
valid for $\text{Re}(x)>0$, $z\in\mathbb{C}$, and for $a(n)$ and $b(n)$ being the arithmetical functions appearing in Definition \ref{class A definition}. 
Since (\ref{popov ryce paper}) generalizes (\ref{Popov intro}), one may ask about the possibility of obtaining a generalization of (\ref{Formula Formula Formula Whittaker GENERAL}) with $r_{k}(n)$ being replaced by a general arithmetical function, $a(n)$. Following the same reasoning as the one present in the proof of our Theorem \ref{generalization whittaker first}, it is possible to obtain such generalization.  
Indeed, if $\phi(s)$ is a Dirichlet series satisfying Definition \ref{class A definition} and $\text{Re}(x)>|\text{Re}(y)|$, $-\frac{r}{2}<\mu<\frac{r}{2}$,
then the following summation formula holds
\begin{align}
-\phi(0)\,\left(2y\right)^{\frac{r}{2}}+\sum_{n=1}^{\infty}&\,a(n)\lambda_{n}^{-\frac{r}{2}}e^{-\lambda_{n}x}\,M_{\mu,\frac{r-1}{2}}\left(2\lambda_{n}y\right)\nonumber \\
=\rho\,\frac{(2y)^{\frac{r}{2}}}{\left(x^{2}+y^{2}\right)^{\frac{r}{2}}}\Gamma(r)+\left(\frac{x-y}{x+y}\right)^{\mu}\,\sum_{n=1}^{\infty}&\,b(n)\,\mu_{n}^{-\frac{r}{2}}\,e^{-\frac{\mu_{n}x}{x^{2}-y^{2}}}M_{-\mu,\frac{r-1}{2}}\left(\frac{2\mu_{n}y}{x^{2}-y^{2}}\right).\label{general summation formula in the Hecke class and Whittaker M}
\end{align}
Note that, as $y\rightarrow0^{+}$ in (\ref{general summation formula in the Hecke class and Whittaker M}), the limiting relation [\cite{NIST}, p. 335, eq. (13.14.14)] $M_{\mu, \nu}(z)=z^{\nu +\frac{1}{2}}\,\left(1+O(z)\right),\,\,\,2\nu \notin \mathbb{Z}^{-},$ $z\rightarrow 0,$ yields
the well-known formula
\begin{equation}
-\phi(0)+\sum_{n=1}^{\infty}a(n)\,e^{-\lambda_{n}x}=\rho\Gamma(r)\,x^{-r}+x^{-r}\,\sum_{n=1}^{\infty}b(n)\,e^{-\frac{\mu_{n}}{x}},\,\,\,\,\,\text{Re}(x)>0,\label{Bochner Popov Paper}
\end{equation}
due to Bochner \cite{bochner_modular_relations}, who proved for the first time the equivalence of (\ref{Bochner Popov Paper})
and the functional equation for $\phi(s)$, (\ref{Hecke Dirichlet series Functional}). Note that (\ref{Bochner Popov Paper}) reduces to the theta transformation formula (\ref{jacobi theta intro popov}) when $\phi(s) = 2\pi^{-s}\,\zeta(2s)$.  

The interested reader can now find the particular summation formulas
arising from the very general identity (\ref{general summation formula in the Hecke class and Whittaker M}). To state a neat identity as an
appetizer for such explorations, one can note that, when $a(n)$ is
Ramanujan's $\tau-$function, $\tau(n)$, then the following formula holds
\[
\sum_{n=1}^{\infty}\frac{\tau(n)}{n^{6}}\,e^{-2\pi nx}M_{\mu,\frac{11}{2}}\left(4\pi ny\right)=\left(\frac{x-y}{x+y}\right)^{\mu}\sum_{n=1}^{\infty}\frac{\tau(n)}{n^{6}}\,e^{-\frac{2\pi nx}{x^{2}-y^{2}}}\,M_{-\mu,\frac{11}{2}}\left(\frac{4\pi ny}{x^{2}-y^{2}}\right),
\]
whenever $-6<\text{Re}(\mu)<6$ and $\text{Re}(x)>|\text{Re}(y)|$.

\section{Proof of Theorem \ref{generalization whittaker first}}

Before giving the proof of Theorem \ref{generalization whittaker first}, we need the following
integral representation for the Whittaker function of the first kind, which can
be found in [\cite{NIST}, p. 337, eq. (13.16.3) and (13.16.4)].
\begin{lemma} \label{Rep Whittaker first kind to prove general popov}
For the confluent hypergeometric function $M_{\rho, \nu}(z)$, the following integral representations hold\footnote{Note that the condition given in \cite{NIST} for the integral (\ref{Second Integral Representation Whittaker M}) contains a typo, which is corrected in the NIST \href{https://dlmf.nist.gov/13.16}{webpage}.}
\begin{equation}
M_{\rho,\nu}(z)=\frac{\Gamma(1+2\nu)\,\sqrt{z}}{\Gamma\left(\frac{1}{2}+\rho+\nu\right)}\,e^{\frac{z}{2}}\,\intop_{0}^{\infty}e^{-t}t^{\rho-\frac{1}{2}}\,J_{2\nu}\left(2\sqrt{zt}\right)\,dt,\,\,\,\,\,\,\text{Re}\left(\rho+\nu\right)+\frac{1}{2}>0,\,\,\,\,z\in\mathbb{C},\label{First integral representation Whittaker M}
\end{equation}
\begin{equation}
M_{\rho,\nu}(z)=\frac{\Gamma(1+2\nu)\,\sqrt{z}}{\Gamma\left(\frac{1}{2}+\rho-\nu\right)}\,e^{-\frac{z}{2}}\,\intop_{0}^{\infty}e^{-t}t^{-\rho-\frac{1}{2}}I_{2\nu}\left(2\sqrt{zt}\right)\,dt,\,\,\,\,\,\,\text{Re}\left(\rho-\nu\right)-\frac{1}{2}<0,\,\,\,\,z\in\mathbb{C}.\label{Second Integral Representation Whittaker M}
\end{equation}
\end{lemma}

This proof's point of departure is the first integral representation (\ref{First integral representation Whittaker M}),
as well as the ubiquitous Popov's formula (\ref{Popov intro}). As remarked above, the advantage of the forthcoming proof is that we no longer need to invoke Watson type
estimates for the Gauss hypergeometric function,
nor the somewhat restrictive conditions that these entail, i.e., the
condition that $x>y>0$. Throughout our next argument we shall assume that $x$ and $y$ are complex
numbers such that $\text{Re}(x)>|\text{Re}(y)|$ and also that $-\frac{k}{4}<\text{Re}(\rho)<\frac{k}{4}$. Under these hypotheses, (\ref{First integral representation Whittaker M}) yields
\begin{align}
\sum_{n=1}^{\infty}\frac{r_{k}(n)}{n^{\frac{k}{4}}}e^{-\pi nx}\,M_{\rho,\frac{k}{4}-\frac{1}{2}}\left(2\pi ny\right) & =\frac{\Gamma\left(\frac{k}{2}\right)\sqrt{2\pi y}}{\Gamma\left(\frac{k}{4}+\rho\right)}\,\sum_{n=1}^{\infty}\frac{r_{k}(n)}{n^{\frac{k}{4}-\frac{1}{2}}}e^{-\pi n\left(x-y\right)}\,\intop_{0}^{\infty}e^{-t}t^{\rho-\frac{1}{2}}J_{\frac{k}{2}-1}\left(2\sqrt{2\pi ny\,t}\right)\,dt\nonumber \\
 & =\frac{\Gamma\left(\frac{k}{2}\right)\sqrt{2\pi y}}{\Gamma\left(\frac{k}{4}+\rho\right)}\,\intop_{0}^{\infty}e^{-t}t^{\rho-\frac{1}{2}}\,\sum_{n=1}^{\infty}\frac{r_{k}(n)}{n^{\frac{k}{4}-\frac{1}{2}}}e^{-\pi n\left(x-y\right)}J_{\frac{k}{2}-1}\left(2\sqrt{2\pi ny\,t}\right)\,dt,\label{Right hand side with Popov to APPPLY}
\end{align}
where the interchange of the integration and summation orders is due
to the simple inequalities
\begin{align}
\sum_{n=1}^{\infty}\frac{r_{k}(n)}{n^{\frac{k}{4}-\frac{1}{2}}}e^{-\pi n\text{Re}(x-y)}\,\intop_{0}^{\infty}e^{-t}t^{\text{Re}(\rho)-\frac{1}{2}}\left|J_{\frac{k}{2}-1}\left(2\sqrt{2\pi ny\,t}\right)\right|\,dt\nonumber \\
\leq\sum_{n=1}^{\infty}\frac{r_{k}(n)}{n^{\frac{k}{4}-\frac{1}{2}}}e^{-\pi n\text{Re}(x-y)}\,\intop_{0}^{\infty}e^{-t}t^{\text{Re}(\rho)-\frac{1}{2}}I_{\frac{k}{2}-1}\left(2\sqrt{2\pi n|y|t}\right)\,dt\nonumber \\
\leq\frac{\left(2\pi|y|\right)^{\frac{k}{4}-\frac{1}{2}}}{\Gamma\left(\frac{k}{2}\right)}\,\sum_{n=1}^{\infty}r_{k}(n)e^{-\pi n\text{Re}(x-y)}\,\intop_{0}^{\infty}e^{-\left(t-2\sqrt{2\pi n|y|t}\right)}t^{\text{Re}(\rho)+\frac{k}{4}-1}\,dt\nonumber \\
=\frac{2\left(2\pi|y|\right)^{\frac{k}{4}-\frac{1}{2}}}{\Gamma\left(\frac{k}{2}\right)}\,\sum_{n=1}^{\infty}r_{k}(n)e^{-\pi n\text{Re}(x-y)}\,\intop_{0}^{\infty}e^{-\left(u^{2}-2\sqrt{2\pi n|y|}\,u\right)}u^{2\text{Re}(\rho)+\frac{k}{2}-1}\,du\nonumber \\
=\frac{2\left(2\pi|y|\right)^{\frac{k}{4}-\frac{1}{2}}}{\Gamma\left(\frac{k}{2}\right)}\frac{\Gamma\left(2\text{Re}(\rho)+\frac{k}{2}\right)}{2^{\text{Re}(\rho)+\frac{k}{4}-\frac{1}{2}}}\,\sum_{n=1}^{\infty}r_{k}(n)e^{-\pi n\text{Re}(x-y)}\,e^{\pi n|y|}\,D_{-2\text{Re}(\rho)-\frac{k}{2}}\left(2\sqrt{\pi n|y|}\right),\label{interchange second proof POPOV}
\end{align}
where the second inequality is due to the integral representation
for the modified Bessel function {[}\cite{NIST}, p.252, eq. (10.32.2){]},
\begin{equation}
\left(\frac{z}{2}\right)^{-\nu}I_{\nu}(z)=\frac{1}{\sqrt{\pi}\Gamma\left(\nu+\frac{1}{2}\right)}\,\intop_{-1}^{1}\left(1-t^{2}\right)^{\nu-\frac{1}{2}}\,e^{zt}dt,\,\,\,\,\,\text{Re}(\nu)>-\frac{1}{2},\,\,\,z\in\mathbb{C},\label{analogue Poisson integral}
\end{equation}
from which one can immediately obtain the bound (with real $\nu>-\frac{1}{2}$
and $z\in\mathbb{C}$)\footnote{despite the fact that the Poisson integral (\ref{analogue Poisson integral}) holds for $\text{Re}(\nu)>-\frac{1}{2}$, which covers the case where $k\geq 2$, when $\nu = -\frac{1}{2}$ ($k=1$), the identity $I_{-\frac{1}{2}}(x)=\sqrt{\frac{2}{\pi x}}\cosh(x)$ actually serves the same purpose.}
\begin{equation}
\left|\left(\frac{z}{2}\right)^{-\nu}I_{\nu}(z)\right|\leq\frac{e^{|\text{Re}(z)|}}{\Gamma(\nu+1)}.\label{Bound Bessel}
\end{equation}
Moreover, the last step of (\ref{interchange second proof POPOV}) can be justified by the integral representation for the parabolic cylinder function {[}\cite{handbook_marichev},
p. 20, 2.2.1.6{]} (see also relation 2.3.15.3 on page 343 of vol. I of \cite{MARICHEV}),
\begin{equation}
\intop_{0}^{\infty}u^{\mu-1}e^{-xu^{2}-yu}du=\frac{\Gamma(\mu)}{(2x)^{\mu/2}}\,e^{\frac{y^{2}}{8x}}\,D_{-\mu}\left(\frac{y}{\sqrt{2x}}\right),\,\,\,\,\text{Re}(\mu),\,\,\text{Re}(x)>0,\,\,y\in\mathbb{C},\label{integral cylinder useful final corollary}
\end{equation}
which can be applied because $\text{Re}(\rho)>-\frac{k}{4}$ by hypothesis.
The last series in (\ref{interchange second proof POPOV}) converges absolutely, as 
 $D_{\mu}(x)$ satisfies the asymptotic
formula (cf. {[}\cite{NIST}, p. 309, 12.9 (i){]}, {[}\cite{ERDELIY_TRANSCENDENTAL}, Vol. II, p. 122, eq. 8.4(1){]}), 
\begin{equation}
D_{\mu}(x)\sim x^{\mu}e^{-\frac{x^{2}}{4}},\,\,\,\,x\rightarrow\infty.\label{Whittaker asymptotic formula}
\end{equation} 
Looking now at the
right-hand side of (\ref{Right hand side with Popov to APPPLY}),
we are able to apply Popov's formula (\ref{Popov intro}) to the infinite series
\[
\sum_{n=1}^{\infty}\frac{r_{k}(n)}{n^{\frac{k}{4}-\frac{1}{2}}}e^{-\pi n\left(x-y\right)}J_{\frac{k}{2}-1}\left(2\sqrt{2\pi ny\,t}\right)
\]
because, by hypothesis, $\text{Re}(x)>|\text{Re}(y)|$. Employing (\ref{Popov intro})
and replacing there $x$ by $x-y$ and $z$ by $\sqrt{\frac{8yt}{x-y}}$,
we obtain the equivalent identity
\begin{align*}
\sum_{n=1}^{\infty}\frac{r_{k}(n)}{n^{\frac{k}{4}-\frac{1}{2}}}e^{-\pi n\left(x-y\right)}J_{\frac{k}{2}-1}\left(2\sqrt{2\pi ny\,t}\right) =\frac{2^{\frac{k}{4}-\frac{1}{2}}\pi^{\frac{k}{4}-\frac{1}{2}}y^{\frac{k}{4}-\frac{1}{2}}}{\Gamma\left(\frac{k}{2}\right)\left(x-y\right)^{\frac{k}{2}}}\,t^{\frac{k}{4}-\frac{1}{2}}e^{-\frac{2yt}{x-y}}&\\
-\frac{\pi^{\frac{k}{4}-\frac{1}{2}}2^{\frac{k}{4}-\frac{1}{2}}y^{\frac{k}{4}-\frac{1}{2}}}{\Gamma\left(\frac{k}{2}\right)}\,t^{\frac{k}{4}-\frac{1}{2}}  +\frac{e^{-\frac{2yt}{x-y}}}{x-y}\,\sum_{n=1}^{\infty}\frac{r_{k}(n)}{n^{\frac{k}{4}-\frac{1}{2}}}\,e^{-\frac{\pi n}{x-y}}\,I_{\frac{k}{2}-1}\left(\frac{\sqrt{8\pi nyt}}{x-y}\right).&
\end{align*}
This sibling of Popov's formula (\ref{Popov intro}) now yields the equality
\begin{align*}
&\frac{\Gamma\left(\frac{k}{2}\right)\sqrt{2\pi y}}{\Gamma\left(\frac{k}{4}+\rho\right)}\,\intop_{0}^{\infty}e^{-t}t^{\rho-\frac{1}{2}}\,\sum_{n=1}^{\infty}\frac{r_{k}(n)}{n^{\frac{k}{4}-\frac{1}{2}}}e^{-\pi n\left(x-y\right)}J_{\frac{k}{2}-1}\left(2\sqrt{2\pi ny\,t}\right)\,dt\\
=&\frac{\left(2\pi y\right)^{\frac{k}{4}}}{\Gamma\left(\frac{k}{4}+\rho\right)}\,\frac{1}{\left(x-y\right)^{\frac{k}{2}}}\,\intop_{0}^{\infty}e^{-\frac{x+y}{x-y}\,t}t^{\rho+\frac{k}{4}-1}\,dt-\frac{\left(2\pi y\right)^{\frac{k}{4}}}{\Gamma\left(\frac{k}{4}+\rho\right)}\,\intop_{0}^{\infty}\,t^{\rho+\frac{k}{4}-1}e^{-t}\,dt\\
+&\frac{\Gamma\left(\frac{k}{2}\right)\sqrt{2\pi y}}{\Gamma\left(\frac{k}{4}+\rho\right)\left(x-y\right)}\,\intop_{0}^{\infty}e^{-\frac{x+y}{x-y}\,t}t^{\rho-\frac{1}{2}}\,\sum_{n=1}^{\infty}\frac{r_{k}(n)}{n^{\frac{k}{4}-\frac{1}{2}}}\,e^{-\frac{\pi n}{x-y}}\,I_{\frac{k}{2}-1}\left(\frac{\sqrt{8\pi nyt}}{x-y}\right)\,dt\\
=&\,\frac{\left(2\pi y\right)^{\frac{k}{4}}\left(x-y\right)^{\rho-\frac{k}{4}}}{\left(x+y\right)^{\rho+\frac{k}{4}}}-\left(2\pi y\right)^{\frac{k}{4}}\\
+&\frac{\Gamma\left(\frac{k}{2}\right)\sqrt{2\pi y}}{\Gamma\left(\frac{k}{4}+\rho\right)\left(x-y\right)}\left(\frac{x-y}{x+y}\right)^{\rho+\frac{1}{2}}\sum_{n=1}^{\infty}\frac{r_{k}(n)}{n^{\frac{k}{4}-\frac{1}{2}}}e^{-\frac{\pi n}{x-y}}\,\intop_{0}^{\infty}e^{-u}\,u^{\rho-\frac{1}{2}}\,I_{\frac{k}{2}-1}\left(\frac{2\sqrt{2\pi nyu}}{\sqrt{x^{2}-y^{2}}}\right)\,du,
\end{align*}
where the evaluation of the first two terms in the last equality can be truly justified by the hypothesis
$-\frac{k}{4}<\text{Re}(\rho)<\frac{k}{4}$. The reason for the interchange of the orders of
summation and integration in the last step of the previous equalities is analogous to that indicated in (\ref{interchange second proof POPOV}).
Indeed,
\begin{align*}
&\sum_{n=1}^{\infty}\frac{r_{k}(n)}{n^{\frac{k}{4}-\frac{1}{2}}}e^{-\pi n\,\text{Re}\left(\frac{1}{x-y}\right)}\,\intop_{0}^{\infty}e^{-u}\,u^{\text{Re}(\rho)-\frac{1}{2}}\,\left|I_{\frac{k}{2}-1}\left(\frac{2\sqrt{2\pi nyu}}{\sqrt{x^{2}-y^{2}}}\right)\right|\,du\\
&\leq \frac{\left(2\pi\right)^{\frac{k}{4}-\frac{1}{2}}}{\Gamma\left(\frac{k}{2}\right)}\,\left|\frac{y}{x^{2}-y^{2}}\right|^{\frac{k}{4}-\frac{1}{2}}\sum_{n=1}^{\infty}r_{k}(n)\,e^{-\pi n\,\text{Re}\left(\frac{1}{x-y}\right)}\,\intop_{0}^{\infty}\,u^{\text{Re}(\rho)+\frac{k}{4}-1}e^{-u}\,e^{2\sqrt{2\pi n}\,\left|\text{Re}\left(\sqrt{\frac{y}{x^{2}-y^{2}}}\right)\right|\,\sqrt{u}}\,du\\
&=\frac{2\left(2\pi\right)^{\frac{k}{4}-\frac{1}{2}}}{\Gamma\left(\frac{k}{2}\right)}\,\left|\frac{y}{x^{2}-y^{2}}\right|^{\frac{k}{4}-\frac{1}{2}}\sum_{n=1}^{\infty}r_{k}(n)\,e^{-\pi n\,\text{Re}\left(\frac{1}{x-y}\right)}\,\intop_{0}^{\infty}\,t^{2\text{Re}(\rho)+\frac{k}{2}-1}e^{-t^{2}}\,e^{2\sqrt{2\pi n}\,\left|\text{Re}\left(\sqrt{\frac{y}{x^{2}-y^{2}}}\right)\right|\,t}\,dt\\
&=\frac{2\left(2\pi\right)^{\frac{k}{4}-\frac{1}{2}}\,\Gamma\left(-2\text{Re}(\rho)-\frac{k}{2}\right)}{\Gamma\left(\frac{k}{2}\right)\,2^{\text{Re}(\rho)+\frac{k}{4}}}\,\left|\frac{y}{x^{2}-y^{2}}\right|^{\frac{k}{4}-\frac{1}{2}}\sum_{n=1}^{\infty}r_{k}(n)\,\exp\left(-\pi n\,\text{Re}\left(\frac{1}{x-y}\right)+ \pi n\left|\text{Re}\left(\sqrt{\frac{y}{x^{2}-y^{2}}}\right)\right|^{2}\right)\\
&\,\,\,\,\,\,\,\,\,\,\,\,\,\,\,\,\,\,\,\,\,\,\,\,\,\,\,\,\,\,\,\,\,\,\,\,\,\,\,\,\,\,\,\,\,\,\,\,\,\,\,\,\,\,\,\,\,\,\,\,\,\,\,\,\,\,\,\,\,\,\,\,\,\,\,\,\,\,\,\,\,\,\,\,\,\,\,\,\,\,\,\,\,\,\,\,\,\,\,\,\,\,\,\,\,\,\,\,\,\,\,\,\,\,\,\,\,\,\,\,\,\,\,\,\,\,\,\times D_{-2\text{Re}(\rho)-\frac{k}{2}}\left(2\sqrt{\pi n}\,\,\left|\text{Re}\left(\sqrt{\frac{y}{x^{2}-y^{2}}}\right)\right|\right),
\end{align*}
which converges absolutely by the asymptotic formula (\ref{Whittaker asymptotic formula}) for the cylinder
function $D_{\nu}(x)$. On the other hand,
using the hypothesis that $-\frac{k}{4}<\text{Re}(\rho)<\frac{k}{4}$, we see from (\ref{Second Integral Representation Whittaker M})
that
\[
\intop_{0}^{\infty}e^{-u}\,u^{\rho-\frac{1}{2}}\,I_{\frac{k}{2}-1}\left(2\sqrt{\frac{2\pi nyu}{x^{2}-y^{2}}}\right)\,du=\frac{\Gamma\left(\rho+\frac{k}{4}\right)}{\Gamma\left(\frac{k}{2}\right)}\sqrt{\frac{x^{2}-y^{2}}{2\pi ny}}e^{\frac{\pi ny}{x^{2}-y^{2}}}\,M_{-\rho,\frac{k}{4}-\frac{1}{2}}\left(\frac{2\pi ny}{x^{2}-y^{2}}\right),
\]
yielding the identity
\begin{align}
\sum_{n=1}^{\infty}\frac{r_{k}(n)}{n^{\frac{k}{4}}}e^{-\pi nx}\,M_{\rho,\frac{k}{4}-\frac{1}{2}}\left(2\pi ny\right) & =\left(\frac{2\pi y}{x^{2}-y^{2}}\right)^{\frac{k}{4}}\,\left(\frac{x-y}{x+y}\right)^{\rho}-\left(2\pi y\right)^{\frac{k}{4}}\nonumber \\
+\left(\frac{x-y}{x+y}\right)^{\rho} & \sum_{n=1}^{\infty}\frac{r_{k}(n)}{n^{\frac{k}{4}}}\,e^{-\frac{\pi nx}{x^{2}-y^{2}}}\,M_{-\rho,\frac{k}{4}-\frac{1}{2}}\left(\frac{2\pi ny}{x^{2}-y^{2}}\right),\label{equation almost at the end Popov general}
\end{align}
which is precisely (\ref{Formula Formula Formula Whittaker GENERAL}), the desired summation formula. $\blacksquare$

\bigskip{}

\bigskip{}

As promised at the beginning of our paper, we can now derive a much stronger version of formulas (\ref{Formula to prove analogue Popov in chapter 6}) and (\ref{formula for modified bessel in chapter 6}), totally erasing its restrictive condition on $x$ and $y$. 
\begin{corollary}
If $x,y$ are two complex numbers such that $\text{Re}(x)>|\text{Im}(y)|$,
then the following transformation formula holds
\begin{align}
\frac{(\pi y)^{\frac{k}{4}-\frac{1}{2}}}{2^{\frac{k}{4}-\frac{1}{2}}\Gamma\left(\frac{k}{4}+\frac{1}{2}\right)}+\sum_{n=1}^{\infty}\frac{r_{k}(n)}{n^{\frac{k}{4}-\frac{1}{2}}}\,e^{-\pi nx}J_{\frac{k}{4}-\frac{1}{2}}\left(\pi ny\right)\nonumber \\
=\frac{(\pi y)^{\frac{k}{4}-\frac{1}{2}}}{2^{\frac{k}{4}-\frac{1}{2}}\Gamma\left(\frac{k}{4}+\frac{1}{2}\right)\left(x^{2}+y^{2}\right)^{\frac{k}{4}}}+\frac{1}{\sqrt{x^{2}+y^{2}}}\,\sum_{n=1}^{\infty}\frac{r_{k}(n)}{n^{\frac{k}{4}-\frac{1}{2}}}\,e^{-\frac{\pi nx}{x^{2}+y^{2}}} & J_{\frac{k}{4}-\frac{1}{2}}\left(\frac{\pi ny}{x^{2}+y^{2}}\right).\label{LOL 1}
\end{align}
Moreover, if $x,y$ are two complex numbers such that $\text{Re}(x)>|\text{Re}(y)|$, the analogous formula is valid
\begin{align}
\frac{(\pi y)^{\frac{k}{4}-\frac{1}{2}}}{2^{\frac{k}{4}-\frac{1}{2}}\Gamma\left(\frac{k}{4}+\frac{1}{2}\right)}+{\sum_{n=1}^{\infty}\frac{r_{k}(n)}{n^{\frac{k}{4}-\frac{1}{2}}}\,e^{-\pi nx}I_{\frac{k}{4}-\frac{1}{2}}\left(\pi ny\right)}\nonumber \\
=\frac{(\pi y)^{\frac{k}{4}-\frac{1}{2}}}{2^{\frac{k}{4}-\frac{1}{2}}\Gamma\left(\frac{k}{4}+\frac{1}{2}\right)\left(x^{2}-y^{2}\right)^{\frac{k}{4}}}+\frac{1}{\sqrt{x^{2}-y^{2}}}\sum_{n=1}^{\infty}\frac{r_{k}(n)}{n^{\frac{k}{4}-\frac{1}{2}}}\,e^{-\frac{\pi nx}{x^{2}-y^{2}}} & I_{\frac{k}{4}-\frac{1}{2}}\left(\frac{\pi ny}{x^{2}-y^{2}}\right).\label{Lol2}
\end{align}
\end{corollary}

\begin{proof}
Clearly, (\ref{Lol2}) implies (\ref{LOL 1}) under the substitution
$y\leftrightarrow iy$. Since the conditions over $x$ and $y$ in
the formula (\ref{Lol2}) are actually the same as the conditions
in our general Theorem \ref{generalization whittaker first}, we just need to check that (\ref{Formula Formula Formula Whittaker GENERAL})
reduces to (\ref{Lol2}). Indeed, if we take $\rho=0$ in (\ref{Formula Formula Formula Whittaker GENERAL})
and use the reduction formula [\cite{NIST}, p. 338, eq. (13.18.8)], 
\[
M_{0,\nu}(2z)=2^{2\nu+\frac{1}{2}}\Gamma(\nu+1)\,\sqrt{z}\,I_{\nu}(z),
\]
we can easily get (\ref{Lol2}).

\end{proof}

\begin{remark}
We can generalize the previous corollary for any complex numbers $x,y$ such that $\text{Re}(x)>|\text{Im}(y)|$. In fact, such generalization reads
\begin{align}
-\frac{\phi(0)}{\Gamma\left(\frac{r+1}{2}\right)}\left(\frac{y}{2}\right)^{\frac{r-1}{2}}+&\sum_{n=1}^{\infty}a(n)\,\lambda_{n}^{\frac{1-r}{2}}e^{-\lambda_{n}x}J_{\frac{r-1}{2}}\left(\lambda_{n}y\right)\nonumber \\
=\frac{\rho\Gamma\left(\frac{r}{2}\right)}{\sqrt{\pi}}\,\frac{(2y)^{\frac{r-1}{2}}}{\left(x^{2}+y^{2}\right)^{\frac{r}{2}}}+\frac{1}{\sqrt{x^{2}+y^{2}}}\,&\sum_{n=1}^{\infty}b(n)\,\mu_{n}^{\frac{1-r}{2}}\exp\left\{ -\frac{\mu_{n}x}{x^{2}+y^{2}}\right\} J_{\frac{r-1}{2}}\left(\frac{\mu_{n}y}{x^{2}+y^{2}}\right).\label{General analogue for any HECKE!}
\end{align}
Analogously, under the condition that $\text{Re}(x)>|\text{Re}(y)|$, one can find the generalization of (\ref{Lol2}), 
\begin{align}
-\frac{\phi(0)}{\Gamma\left(\frac{r+1}{2}\right)}\left(\frac{y}{2}\right)^{\frac{r-1}{2}}+&\sum_{n=1}^{\infty}a(n)\,\lambda_{n}^{\frac{1-r}{2}}e^{-\lambda_{n}x}I_{\frac{r-1}{2}}\left(\lambda_{n}y\right)\nonumber \\
=\frac{\rho\Gamma\left(\frac{r}{2}\right)}{\sqrt{\pi}}\,\frac{(2y)^{\frac{r-1}{2}}}{\left(x^{2}-y^{2}\right)^{\frac{r}{2}}}+\frac{1}{\sqrt{x^{2}-y^{2}}}\,&\sum_{n=1}^{\infty}b(n)\,\mu_{n}^{\frac{1-r}{2}}\exp\left\{ -\frac{\mu_{n}x}{x^{2}-y^{2}}\right\} I_{\frac{r-1}{2}}\left(\frac{\mu_{n}y}{x^{2}-y^{2}}\right).\label{General analogue for any HECKE!-1}
\end{align}
\end{remark}

\bigskip{}
\bigskip{}
\bigskip{}

We can also present a very interesting formula involving the Laguerre function. The next corollary is a very curious case of Theorem \ref{generalization whittaker first} happening when $k=2$. 
\begin{corollary}
For $-\frac{1}{2}<\text{Re}(\rho)<\frac{1}{2}$ and complex $x,y$ such that $\text{Re}(x)>|\text{Re}(y)|$, the following summation formula holds
\begin{align}
&1+\sum_{n=1}^{\infty}r_{2}(n)e^{-\pi n(x+y)}\,L_{\rho-\frac{1}{2}}\left(2\pi ny\right)\nonumber \\
=&\frac{1}{\sqrt{x^{2}-y^{2}}}\left(\frac{x-y}{x+y}\right)^{\rho}+\frac{1}{\sqrt{x^{2}-y^{2}}}\left(\frac{x-y}{x+y}\right)^{\rho}\,\sum_{n=1}^{\infty}r_{2}(n)e^{-\frac{\pi n}{x-y}}\,L_{-\rho-\frac{1}{2}}\left(\frac{2\pi ny}{x^{2}-y^{2}}\right),\label{Laguerre reduction formula!}
\end{align}
where $L_{\mu}(z)$ denotes the Laguerre function.
\end{corollary}

\begin{proof}
The proof is just an application of (\ref{Formula Formula Formula Whittaker GENERAL})
with $k=2$ and the reduction formula for the Whittaker $M-$function
\[
M_{\rho,0}(z)=\sqrt{z}\,e^{-\frac{z}{2}}\,L_{\rho-\frac{1}{2}}\left(z\right).
\]
\end{proof}

\bigskip{}

We end this paper by presenting yet another interesting corollary, whose formal shape is strikingly similar to (\ref{formula for modified bessel in chapter 6}). 

\begin{corollary}

Let $x,y$ be two complex numbers such that $\text{Re}(x)>|\text{Re}(y)|$ and $k\geq 3$.
Then the following transformation formula holds
\begin{align}
(2\pi y)^{\frac{k}{4}}\sqrt{\frac{x-y}{x+y}}+2^{\frac{k}{2}-1}\Gamma\left(\frac{k}{4}\right)\pi y\sqrt{\frac{x-y}{x+y}}\,\sum_{n=1}^{\infty}\frac{r_{k}(n)}{n^{\frac{k}{4}-1}} & e^{-\pi nx}\,\left(I_{\frac{k}{4}}\left(\pi ny\right)+I_{\frac{k}{4}-1}\left(\pi ny\right)\right)\nonumber \\
=\left(\frac{2\pi y}{x^{2}-y^{2}}\right)^{\frac{k}{4}}+2^{\frac{k}{2}-1}\pi\,\Gamma\left(\frac{k}{4}\right)\,\frac{y}{x^{2}-y^{2}}\,\sum_{n=1}^{\infty}\frac{r_{k}(n)}{n^{\frac{k}{4}-1}} & e^{-\frac{\pi nx}{x^{2}-y^{2}}}\,\left(I_{\frac{k}{4}-1}\left(\frac{\pi ny}{x^{2}-y^{2}}\right) - I_{\frac{k}{4}}\left(\frac{\pi ny}{x^{2}-y^{2}}\right)\right).\label{Third corolllary!}
\end{align}
\end{corollary}

\begin{proof}
Taking $\rho=-\frac{1}{2}$ in the general formula (\ref{Formula Formula Formula Whittaker GENERAL})
and using the reduction formulas for the Whittaker function,
\[
M_{-\frac{1}{2},\nu}\left(z\right)=2^{2\nu-1}\Gamma\left(\nu+\frac{1}{2}\right)\,z\,\left(I_{\nu+\frac{1}{2}}\left(\frac{z}{2}\right)+I_{\nu-\frac{1}{2}}\left(\frac{z}{2}\right)\right),
\]
\[
M_{\frac{1}{2},\nu}(z)=2^{2\nu-1}\Gamma\left(\nu+\frac{1}{2}\right)\,z\,\left(I_{\nu-\frac{1}{2}}\left(\frac{z}{2}\right)-I_{\nu+\frac{1}{2}}\left(\frac{z}{2}\right)\right),
\]
we get, for $k\geq 3$,
\begin{align*}
(2\pi y)^{\frac{k}{4}}\sqrt{\frac{x-y}{x+y}}+2^{\frac{k}{2}-1}\Gamma\left(\frac{k}{4}\right)\pi y\sqrt{\frac{x-y}{x+y}}\,\sum_{n=1}^{\infty}\frac{r_{k}(n)}{n^{\frac{k}{4}-1}} & e^{-\pi nx}\,\left(I_{\frac{k}{4}}\left(\pi ny\right)+I_{\frac{k}{4}-1}\left(\pi ny\right)\right)\\
=\left(\frac{2\pi y}{x^{2}-y^{2}}\right)^{\frac{k}{4}}+2^{\frac{k}{2}-1}\pi\,\Gamma\left(\frac{k}{4}\right)\,\frac{y}{x^{2}-y^{2}}\,\sum_{n=1}^{\infty}\frac{r_{k}(n)}{n^{\frac{k}{4}-1}} & e^{-\frac{\pi nx}{x^{2}-y^{2}}}\,\left(I_{\frac{k}{4}-1}\left(\frac{\pi ny}{x^{2}-y^{2}} \right) - I_{\frac{k}{4}}\left(\frac{\pi ny}{x^{2}-y^{2}}\right)\right). 
\end{align*}

\end{proof}
\begin{remark}
It is actually possible to get analogues of the summation formula (\ref{Third corolllary!})
by specifying $\rho=-\frac{\ell}{2}$, $\ell\in\mathbb{N}$, in (\ref{Formula Formula Formula Whittaker GENERAL}).
But the resulting formula is somewhat cumbersome, so we leave
the details of this to the reader.
\end{remark}

\end{document}